\documentclass[a4paper,12pt,reqno]{amsart}

\usepackage{amsmath}
\usepackage{amscd}
\usepackage{amssymb}
\usepackage{mathrsfs}
\usepackage[left=2.5cm,right=2.5cm,bottom=3cm,top=3cm]{geometry}
\newtheorem{thm}{Theorem}[section]
\newtheorem{cor}[thm]{Corollary}

\newtheorem{lem}[thm]{Lemma}

\theoremstyle{definition}

\newtheorem{rem}[thm]{Remark}

\numberwithin{equation}{section}

\newcommand{\refl}[1]{Lemma~{\rm \ref{#1}}}
\newcommand{\reft}[1]{Theorem~{\rm \ref{#1}}}
\newcommand{\refe}[1]{(\ref{#1})}
\newcommand{\diff}{\,\mathrm{d}}
\newcommand{\dd}{\sqrt{-1}\partial\bar{\partial}}
\begin{document}

\title[Optimal geometric estimates for compact K\"ahler manifolds]{Optimal geometric estimates for compact K\"ahler manifolds of a Nash entropy bound}

\author{Weiqi Zhang and Yashan Zhang}

\date{\today}
\address{Hunan University, Changsha, China}

\email{weiqizhang@hnu.edu.cn}
\email{yashanzh@hnu.edu.cn}

\begin{abstract}
We prove Sobolev-type inequality and local volume noncollapsing with optimal exponents for compact K\"ahler manifolds of uniformly bounded $q$-Nash entropy.
\end{abstract}

\maketitle

\section{Introduction}
\subsection{Background and motivation} The past decades have seen extensive developments on the geometry and analysis of Riemannian manifolds under Ricci curvature lower bound (see e.g. \cite{Ch,Li}), in which one of the most fundamental tools is the relative volume comparison, in view of applying Gromov's precompactness theorem. Essentially, as explained in \cite[page 24-25]{GPSS1}, to obtain precompactness for closed manifolds it may suffice to have a growth estimate on relative volumes of geodesic balls, which is called local volume non-collapsing estimate. Another most fundamental role is played by Sobolev-type inequality, which holds with the Sobolev constant usually being controlled by the Ricci curvature lower bound and implies a variety of geometric estimates (see e.g. \cite{H,Li}). Of course this theory on Ricci curvature lower bound also applies with great success in K\"ahler geometry. However, the studies of degeneration of canonical K\"ahler metrics and singularity formation of the K\"ahler-Ricci flow in K\"ahler geometry require compactness and geometric estimates on natural families of K\"ahler manifolds, where the Ricci curvature may not have a uniform global lower bound. These provide a strong motivation to develop geometric estimates on K\"ahler manifolds without involving Ricci curvature bounds. In a recent remarkable series of works, Guo, Phong, Song and Sturm \cite{GPS,GPSS1,GPSS2,GPSS3} systematically established geometric estimates for compact K\"ahler manifolds of bounded $q$-Nash entropy, particularly including local volume non-collapsing estimate and Sobolev-type inequality. The strikng new feature in their works is that no Ricci curvature lower bound is involved. Further related developments can be found in \cite{DNV,GGZ,GT,Liuj,NV,V,ZlZzl,ZwZy} etc. 

Our goal in this paper is to prove the optimal Sobolev-type inequality and optimal local volume non-collapsing estimate under the $q$-Nash entropy condition. To be more precise, we now introduce necessary notations. Let $(X,\omega_X)$ be an $n$-dimensional compact K\"ahler manifold ($n\ge2$) and $\mathcal{K}(X)$ be the space of K\"ahler metrics on $X$.  We define the $q$-Nash entropy of a K\"ahler metric $\omega$ on $X$ with respect to $\omega_X$ by 
\begin{equation}
    \mathcal{N}_{X,\omega_X,q}(\omega)=\int_X e^F\log^q(1+e^F)\omega_X^n=\Vert e^F\Vert_{L^1 \log L^q(X,\omega_X)},\nonumber
\end{equation} 
where $e^F=\frac{1}{V_\omega} \frac{\omega^n}{\omega^n_X}$ and $V_\omega=\int_X \omega^n$. Following \cite{GPSS1}, we introduce the following set of K\"ahler metrics for given parameters $A,K>0$ and $q>n$:
\begin{equation}
    \mathcal{V}(X,\omega_X,n,A,q,K):=\left\{\omega\in\mathcal{K}(X):\int_X \omega\wedge\omega_X^{n-1}\leq A,\mathcal{N}_{X,\omega_X,q}(\omega)\leq K\right\}.\nonumber
\end{equation}
With these notations, the above-mentioned results in \cite{GPSS1,GPSS2,GPSS3,GT,V,NV} state that for any (small) number $\epsilon>0$, there are two positive numbers $C_\epsilon$ and $c_\epsilon$, depending on $X,\omega_X,n,A,q,K$ and $\epsilon$, such that for any K\"ahler metric $\omega\in \mathcal{V}(X,\omega_X,n,A,q,K)$ the followings hold:\\
(1)  local volume non-collapsing:
 \begin{equation}\label{eq-vol-epsilon}
    \frac{V_\omega(B_\omega(x,R))}{V_\omega(X)}\ge c_\epsilon R^{\frac{2nq}{q-n}+\epsilon},\,\,\forall x\in X,\,\,\,\forall R\in(0,1];
    \end{equation}
(2) Sobolev-type inequality:
 \begin{equation}\label{eq-sob-epsilon}
      \left(\frac{1}{V_\omega}\int_X |u-\bar{u}|^{\frac{2nq}{nq-q+n}-\epsilon}\omega^n\right)^{\frac{1}{\frac{2nq}{nq-q+n}-\epsilon}}\le C_\epsilon\left(\frac{1}{V_\omega}\int_X |\nabla u|^2_\omega \omega^n\right)^{1/2}.
    \end{equation}
    Actually, \eqref{eq-sob-epsilon} implies \eqref{eq-vol-epsilon} as well as uniform 
     estimates on heat kernel, eigenvalues etc., see \cite{GPSS3,Li}.
     
     Given the vital importance of local volume non-collapsing and Sobolev-type inequality, one naturally wonders that is it possible to extend the estimates \eqref{eq-vol-epsilon} and \eqref{eq-sob-epsilon} to optimal versions? Note that, under the $q$-Nash entropy condition, the possible optimal version of local volume non-collapsing is the \eqref{eq-vol-epsilon} with $\epsilon=0$, thanks to the explicit example in \cite[Example 3.1]{GS}, and hence the possible optimal version of Sobolev-type inequality is the \eqref{eq-sob-epsilon} with $\epsilon=0$. Moreover, the optimal local volume non-collapsing (i.e. \eqref{eq-vol-epsilon} with $\epsilon=0$) was achieved in \cite{GS} by additionally allowing the constant to be dependent on a (local) lower bound of Ricci curvature. In the general case of no Ricci curvature lower bound being involved, in our previous work \cite[Proposition 5.1, Remark 5.4]{ZwZy}, by further developing the idea in \cite{GPSS1}, the local volume non-collapsing \eqref{eq-vol-epsilon} has been improved to, for any $\tilde\epsilon>0$,
     $$ \frac{V_\omega(B_\omega(x,R))}{V_\omega(X)}\ge\frac{c_{\tilde\epsilon} R^{\frac{2nq}{q-n}}}{(1-\log R)^{\frac{2nq}{q-n}+\tilde\epsilon}}.$$
     
     \subsection{Main results}
     In this paper, we prove that, under the $q$-Nash entropy condition, the optimal Sobolev-type inequality and optimal local volume non-collapsing actually hold in general. 
     
     \begin{thm}\label{Sobolev inequality}
    Let $X$ be an $n$-dimensional compact K\"ahler manifold equipped with a K\"ahler metric $\omega_X$. For any $A,K>0$, $q>n$, there exist two positive numbers $C=C(X,\omega_X,n,A,q,K)$ and $c=c(X,\omega_X,n,A,q,K)$ such that for any $\omega\in\mathcal{V}(X,\omega_X,n,A,q,K)$ there hold:
    \begin{itemize}
    \item[(1)] (Sobolev-type inequality)
    \begin{equation}\label{optimal-sob}
        \left(\frac{1}{V_\omega}\int_X |u-\bar{u}|^{\frac{2nq}{nq-q+n}}\omega^n\right)^{\frac{nq-q+n}{2nq}}\le C\left(\frac{1}{V_\omega}\int_X |\nabla u|^2_\omega \omega^n\right)^{\frac12},\,\,\,\forall \,u\in W^{1,2}(X);
    \end{equation}
    \item[(2)] (local volume non-collapsing)
     \begin{equation}\label{optimal-vol}
        \frac{V_\omega(B_\omega(x,R))}{V_\omega(X)}\ge c R^{\frac{2nq}{q-n}}, \,\,\forall x\in X,\,\,\,\forall R\in(0,1].
        \end{equation}
    \end{itemize}
\end{thm}
     
 The optimality is explained in Remark \ref{rem-opt} with more details. 
 
 Given Sobolev-type inequality in \eqref{optimal-sob}, we can immediately derive the local volume non-collapsing estimate in \eqref{optimal-vol}, see e.g. \cite{H}, \cite[Theorem 4.1.2]{Z} and \cite[Proposition 9.1]{GPSS3}. 
 
Further, several more geometric consequences can be obtained by directly applying \eqref{optimal-sob}. For a compact K\"ahler manifold $(X,\omega)$, let $G_\omega$ and $H_\omega$ be Green's function and heat kernel of $(X,\omega)$ defined by
$$\Delta_\omega G_\omega(x,\cdot)=-\delta_x(\cdot)+V_\omega^{-1}\,\,\,\mathrm{and}\,\,\,\int_XG_\omega(x,\cdot) \omega^n=0,$$
and 
\begin{equation}
    \frac{\partial}{\partial t} H_\omega(x,y,t)=\Delta_{\omega,y}H_\omega(x,y,t)\,\,\, \mathrm{and}\,\,\, \lim_{t\to 0^+} H_\omega(x,y,t)=\delta_x(y)\nonumber,
\end{equation} 
respectively, where $\Delta_\omega$ is the Laplacian operator associated to $\omega$ and $x,y\in X$; write 
$$0=\lambda_0<\lambda_1\le\lambda_2\le\cdots$$
for the increasing sequence of eigenvalues of $-\Delta_\omega$ repeated according to their multiplicity, and then correspondingly choose eigenfunctions $\phi_k$ with eigenvalue $\lambda_k$ such that $\{\phi_k\}$ are orthonormal in $L^2(X,\omega^n)$.

Given Sobolev-type inequality in \eqref{optimal-sob}, applying Cheng-Li's classical arguments \cite{CL} (also see \cite{GPSS3,Li,Z}) immediately gives
 
\begin{cor}\label{cor}
Let $X$ be an $n$-dimensional compact K\"ahler manifold equipped with a K\"ahler metric $\omega_X$. For any $A,K>0$, $q>n$, there exist two positive numbers $C=C(X,\omega_X,n,A,q,K)$ and $c=c(X,\omega_X,n,A,q,K)$ such that for any $\omega\in\mathcal{V}(X,\omega_X,n,A,q,K)$, there hold:
    \begin{itemize}
    \item[(a)] (On-diagonal estimate for heat kernel)
$$H_\omega(x,x,t)\le\frac{1}{V_\omega}+\frac{C}{V_\omega t^{\frac{nq}{q-n}}},\,\,\,\forall x\in X, \,\forall t>0.$$

\item[(b)] (Growth estimate for higher eigenvalue) for any $k\ge1$, 
$$\lambda_k\ge c k^{\frac{q-n}{nq}}.$$

\item[(c)] (Eigenfunction estimate) for any $k\ge1$, 
$$|\phi_k|_{L^\infty(X)}^2\le \frac{C}{V_\omega} \lambda_k^{\frac{nq}{q-n}}.$$
	\end{itemize}
\end{cor}

The exponent of $t$ in Corollary \ref{cor} (a) is optimal (for small $t$), see Remark \ref{rem-opt}.

\subsection{About the proof}
It suffices to prove Sobolev-type inequality stated in Theorem \ref{Sobolev inequality} (1). To this end, we first prove a general result (see \reft{thm-green-sob}) in Section \ref{sect-sob} that, for a compact K\"ahler manifold $(X,\omega)$ (or a closed Riemannian manifold) and $p>2$, an $L^{\frac{p}{2},\infty}$ (i.e. weak $L^{\frac{p}{2}}$) bound of Green function $G_\omega$ implies a Sobolev-type inequality of exponent $p$ and vice versa; this fits the circle of idea firstly developed by Carron \cite[0.7. Theoreme principal]{C} for a noncompact nonparabolic complete Riemannian manifold of infinite volume (also see \cite[Chapter 8, Theorem 8.2]{H}) and further extended by Tian-Wang \cite[Lemma 2.1]{TW} for an Euclidean bounded domain (also see \cite{Le,WZ}).

We then prove in Section \ref{sect-green} a uniform $L^{\frac{nq}{nq-q+n},\infty}$ estimate for Green function of $\omega\in\mathcal{V}(X,\omega_X,n,A,q,K)$ (see \reft{thm-green}), which is inspired by method in \cite{TW,WZ} handling the linearized Monge-Amp\`ere equations of $L^\infty$ volume density over an Euclidean bounded domain and the PDE-based method involving auxiliary complex Monge-Amp\`ere equation developed in \cite{GPSS1,GPSS2}.

By combining Theorems \ref{thm-green-sob} and \ref{thm-green}, we conclude Theorem \ref{Sobolev inequality} (1), and consequently Theorem \ref{Sobolev inequality} (2) and Corollary \ref{cor} follow.

\section{Weak integrability of Green's function is equivalent to Sobolev-type inequality}\label{sect-sob}
In this section we present a general result as follows, which may be of independent interest.
\begin{thm}\label{thm-green-sob}
Let $X$ be an $n$-dimensional compact K\"ahler manifold, $\omega$ a K\"ahler metric on $X$ and $p>2$. The followings are equivalent:
\begin{enumerate}
    \item  There is a constant $L>0$ such that for any $x\in X$
    \begin{equation}\label{weak-int} 
        \Vert G_\omega(x,\cdot)V_\omega\Vert_{L^{p/2,\infty} (X,\frac{\omega^n}{V_\omega})}:=\sup_{t>0}\left\{t^{p/2}\mu_\omega(\{y\in X:|G_\omega(x,y)V_\omega|>t\})\right\}\le L,
    \end{equation}
    where $\mu_\omega(U):=\frac{1}{V_\omega}\int_U \omega^n$ for any subset $U\subset X$.
    \item There is a constant $C_S>0$ such that
    \begin{equation}\label{Sob ineq}
        \left(\frac{1}{V_\omega}\int_X |u-\bar{u}|^{p}\omega^n\right)^{\frac{1}{p}}\le C_S\left(\frac{1}{V_\omega}\int_X |\nabla u|^2_\omega \omega^n\right)^{\frac12},\,\,\,\forall \,u\in W^{1,2}(X).
    \end{equation}
\end{enumerate}
\end{thm}
Here, the equivalence is in the following sense: if (1) holds, then $(2)$ holds with $C_S=C_S(p,L)$; and if (2) holds, then (1) holds with $L=L(p,C_S)$.
\subsection{Proof of \reft{thm-green-sob}, Part One: $(1)\Rightarrow (2)$.} Given (1), we will prove that there exists $C=C(p,L)>0$ such that for any $u\in W^{1,2}(X)$ satisfying $\int_X u\omega^n=0$, the following holds
    \begin{equation}\label{Sobolev ineq}
        \left(\int_X |u|^p \frac{\omega^n}{V_\omega}\right)^{\frac1p}\le C\left(\int_X |\nabla u|^2_\omega \frac{\omega^n}{V_\omega}\right)^{\frac12}.
    \end{equation} 
    
    \subsection*{Step 1: Relating to an eigenvalue problem}
    For any $T>0$, we define
    $\mathcal{F}(s)=\int_0^s f(t)\diff t$, where
    \[
        f(t)=
        \begin{cases}
            -T^{p-1}, &t\le -T;\\
            t|t|^{p-2}, &|t|<T;\\
            T^{p-1}, &t\ge T,
        \end{cases}\]  
    and set
    \begin{equation}\label{s*}
        s^*=\inf\left\{\int_X |\nabla v|_\omega^2 \frac{\omega^n}{V_\omega}: v\in H^1(X), \int_X \mathcal{F}(v)\frac{\omega^n}{V_\omega}=1,\int_X v \frac{\omega^n}{V_\omega}=0\right\}.
    \end{equation}
    Clearly what we need is a positive lower bound on $s^*$, which has to be independent on $T$.
    To begin with we may assume there exists $T>0$ such that $s^*\le 1$. Otherwise for any $u\in C^1(X)$ satisfying $\int_X u\omega^n=0$, we let $\tilde{u}=\lambda u$. We take $T>\sup\tilde{u}$, then 
    \begin{equation}
        \mathcal{F}(\tilde{u})=\int_{0}^{\tilde{u}} t|t|^{p-2}\diff t=\frac{1}{p}|\tilde{u}|^p=\frac{\lambda^p}{p}|u|^p.\nonumber
    \end{equation}
    Integral over $X$ with respect to the volume element $\frac{\omega^n}{V_\omega}$
    \begin{equation}
        \int_X \mathcal{F}(\tilde{u}) \frac{\omega^n}{V_\omega}=\frac{\lambda^p}{p}\Vert u\Vert^p_{L^p}=1,\nonumber
    \end{equation}
    where $\Vert\cdot\Vert_{L^p}=\Vert\cdot\Vert_{L^p(X,\frac{\omega^n}{V_\omega})}$ and we have chosen $\lambda=p^{\frac1p}\Vert u\Vert^{-1}_{L^p}$. From
    \begin{equation}
        \int_X |\nabla\tilde{u}|_\omega^2 \frac{\omega^n}{V_\omega}\ge s^*\ge 1\nonumber
    \end{equation}
    we have 
    \begin{equation}
        \int_X |\nabla u|^2 \frac{\omega^n}{V_\omega} \ge \lambda^{-2}=p^{-2/p}\Vert u\Vert_{L^p}^2.\nonumber
    \end{equation}
    Since $C^1(X)$ is dense in $H^1(X)$, \refe{Sobolev ineq} holds.\\
    
    Next, under the assumption $s^*\le 1$, we intruduce two constant Lagrange multipliers $\lambda, \mu$ and form the following unconstrained functional 
    \[\mathcal{L}(v,\lambda,\mu)=\int_X |\nabla v|^2_\omega \frac{\omega^n}{V_\omega}-\lambda\left(\int_X \mathcal{F}(v)\frac{\omega^n}{V_\omega}-1\right)-\mu\int_X v \frac{\omega^n}{V_\omega}.\]
    Let $v$ be the function achieving the infimum in \refe{s*} (whose existence and regularity is given in \refl{existence and regularity of v} below; in particular, $v$ is $C^{2,\alpha}$). Then there holds
    \begin{align*}
        0=\delta\mathcal{L}&=2\int_X \langle\nabla v,\nabla(\delta v)\rangle_\omega \frac{\omega^n}{V_\omega}-\lambda\int_X f(v)\delta v\frac{\omega^n}{V_\omega}-\mu\int_X \delta v \frac{\omega^n}{V_\omega}\\
        &=-2\int_X \delta v\Delta_\omega v \frac{\omega^n}{V_\omega}-\lambda\int_X \delta v f(v)\frac{\omega^n}{V_\omega}-\mu\int_X \delta v \frac{\omega^n}{V_\omega},
    \end{align*}
   and hence $v$ satisfies
    \begin{equation}\label{v1}
        2\Delta_\omega v=-\lambda f(v)-\mu.
    \end{equation}
    Define $\hat{\lambda}=\frac{\lambda}{2}$. Then by $\int_X \Delta_\omega v\omega^n=0$, we have 
    \begin{equation}\label{v}
        \begin{cases}
            \Delta_\omega v=-\hat{\lambda}f(v)+\hat{\lambda}\int_X f(v)\frac{\omega^n}{V_\omega}\\
            \int_X v \frac{\omega^n}{V_\omega}=0\\
            \int_X \mathcal{F}(v)\frac{\omega^n}{V_\omega}=1
        \end{cases}
    \end{equation} 
    and 
    \begin{equation*}
        s^*=\int_X |\nabla v|^2_\omega \frac{\omega^n}{V_\omega}=-\int_X v\Delta_\omega v \frac{\omega^n}{V_\omega}=\hat{\lambda}\int_X f(v)v \frac{\omega^n}{V_\omega}.\nonumber
    \end{equation*}
    Note that for any $t\in \mathbb{R}$,
    \[\frac{1}{p}tf(t)\le\mathcal{F}(t)\le tf(t),\]
    then we have 
    \begin{equation}\label{s-lambda}
        \hat{\lambda}\le s^*\le p\hat{\lambda}.
    \end{equation}

    \subsection*{Step 2: Faber-Krahn type inequality on the first eigenvalue} 
    For any subset $U$, define $\mu_\omega(U)=\int_U \frac{\omega^n}{V_\omega}$. By assumption, for any $t>0$,
    \begin{equation}\label{weak p/2}
        \mu_\omega(\{y\in X:|G_\omega(x,y)V_\omega|>t\})\le L t^{-\frac{p}{2}}.
    \end{equation}
    For any open subset $U\subset X$, we denote
    \begin{equation}\label{lambda(U)}
        \lambda(U)=\inf\left\{\int_X |\nabla \psi|^2_\omega \frac{\omega^n}{V_\omega}: \psi\in H^1(X), \int_U \psi^2 \frac{\omega^n}{V_\omega}=1,\int_X \psi \omega^n=0\right\}.
    \end{equation}
    Let $\phi$ be the function such that \refe{lambda(U)} attains the infimum (whose existence and regularity is given in \refl{existence and regularity of psi} below). Similarly, $\phi$ satisfies
    \begin{equation}\label{eigenvalue equation}
        \begin{cases}
            \Delta_\omega \phi=-\lambda(U)\phi\chi_U+\lambda(U)\int_U \phi \frac{\omega^n}{V_\omega}\\
            \int_X\phi\omega^n=0.
        \end{cases}
    \end{equation}
    Applying Green's formula, we obtain
    \begin{equation}
        \phi(z)=-\int_X G_\omega(z,\cdot)\Delta_\omega \phi \omega^n=\lambda(U)\int_U G_\omega(z,\cdot)\phi \omega^n.\nonumber
    \end{equation}
    Normalize $\phi$ so that $\Vert \phi\Vert_{L^{\infty}(X)}=1$ and then fix $w\in X$ with $|\phi(w)|=1$. Then we obtain
    \begin{equation}
        1\le\lambda(U)\int_U |G_\omega(w,\cdot)||\phi|\omega^n\le \lambda(U)\int_U|G_\omega(w,\cdot)V_\omega| \frac{\omega^n}{V_\omega}.\nonumber
    \end{equation}
    By \refe{weak p/2}, we have 
    \begin{align}\label{integral in U}
        \int_U |G_\omega(w,\cdot)V_\omega|\frac{\omega^n}{V_\omega}&=\int_0^\infty \mu_\omega(\{z\in U: |G_\omega(w,z)V_\omega|>t\})]\diff t\notag\\
        &\le \int_0^\infty \min\{\mu_\omega(U),L t^{-\frac{p}{2}}\}\diff t\notag\\
        &=T_0\mu_\omega(U)+L\int_{T_0}^\infty t^{-\frac{p}{2}}\diff t\notag\\
        &\le \left(1+\frac{2}{p-2}\right)L^{2/p}\left(\mu_\omega(U)\right)^{1-\frac{2}{p}}
    \end{align}
    where $T_0=\left(\frac{L}{\mu_\omega(U)}\right)^{2/p}$. Then 
    \begin{equation}\label{c}
        c^*:=\inf_{U\subset X}\left\{\lambda(U)\left(\mu_\omega(U)\right)^{1-\frac{2}{p}}\right\}\ge (1-2/p)L^{-\frac{2}{p}}> 0.
    \end{equation}
    \subsection*{Step 3: Measure estimate on the super-level set of $v$} Assume that  $M:=\Vert v\Vert_{L^\infty}$ is attained at some point where $v$ is positive (otherwise replace $v$ by $-v$). Denote $\Omega_t=\{v>M-t\}$ for $0\le t\le M$. We claim that
    \begin{equation}\label{level set of v}
        \mu_\omega(\Omega_t)\ge \min\left\{\frac12,C_p\left(\frac{c^* t}{s^*M^{p-1}}\right)^{\frac{p}{p-2}}\right\},
    \end{equation}
    where $C_p$ is a positive constant depending only on $p$.
    We take $\psi=(v-M+t)_+-\int_{\Omega_t}(v-M+t)\frac{\omega^n}{V_\omega}$ and $\bar{\psi}=\frac{\psi}{\Vert\psi\Vert_{L^2(\Omega_t,\frac{\omega^n}{V_\omega})}}$. Then 
    \begin{equation}
        \lambda(\Omega_t)\le \int_X |\nabla\bar{\psi}|_\omega^2 \frac{\omega^n}{V_\omega}=\frac{\int_X |\nabla \psi|_\omega^2 \frac{\omega^n}{V_\omega}}{\int_{\Omega_t}\psi^2 \frac{\omega^n}{V_\omega}}.\nonumber
    \end{equation}
    The numerator term satisfies
    \begin{align*}
        \int_X |\nabla\psi|^2 \frac{\omega^n}{V_\omega}&=\int_{\Omega_t}|\nabla (v-M+t)|_\omega^2 \frac{\omega^n}{V_\omega}\\
        &=-\int_{\Omega_t}(v-M+t)\Delta_\omega v \frac{\omega^n}{V_\omega}\\
        &=\int_{\Omega_t}(v-M+t)(\hat{\lambda}f(v)-\hat{\lambda}\int_X f(v)\frac{\omega^n}{V_\omega})\\
        &\le 2s^* M^{p-1}\int_{\Omega_t}(v-M+t)\frac{\omega^n}{V_\omega}.
    \end{align*}
    The donominator term satisfies
    \begin{align*}
        \int_{\Omega_t} \psi^2 \frac{\omega^n}{V_\omega}&=\int_{\Omega_t}(v-M+t)^2 \frac{\omega^n}{V_\omega}+(-2+\mu_\omega(\Omega_t))\left(\int_{\Omega_t}(v-M+t)\frac{\omega^n}{V_\omega}\right)^2\\
        &\ge \left(1-\mu_\omega(\Omega_t)\right)^2 \int_{\Omega_t}(v-M+t)^2 \frac{\omega^n}{V_\omega}. \quad \text{(by H\"older's inequality)}
    \end{align*}
    If $\mu_\omega(\Omega_t)\le \frac{1}{2}$, then 
    \begin{align*}
        \lambda(\Omega_t)&\le \frac{2s^* M^{p-1}}{\left(1-\mu_\omega(\Omega_t)\right)^2}\frac{\int_{\Omega_t}(v-M+t)\frac{\omega^n}{V_\omega}}{\int_{\Omega_t}(v-M+t)^2 \frac{\omega^n}{V_\omega}}\\
        &\le 8s^* M^{p-1}\frac{\int_{\Omega_t}(v-M+t)\frac{\omega^n}{V_\omega}}{\int_{\Omega_t}(v-M+t)^2 \frac{\omega^n}{V_\omega}}\\
        &\le 8s^* M^{p-1}\left(\frac{\mu_\omega(\Omega_t)}{\int_{\Omega_t}(v-M+t)^2 \frac{\omega^n}{V_\omega}}\right)^{1/2}\\
        &\le 8s^* M^{p-1}\left(\frac{\mu_\omega(\Omega_t)}{\left(\frac{t}{2}\right)^2\mu_{\omega}(\Omega_{t/2})}\right)^{1/2}
    \end{align*}
    By \refe{c} 
    \[c^*\mu_\omega(\Omega_t)^{\frac{2}{p}-1}\le \lambda(\Omega_t)\le\frac{16s^* M^{p-1}}{t}\left(\frac{\mu_\omega(\Omega_t)}{\mu_{\omega}(\Omega_{t/2})}\right)^{1/2},\]
    i.e.
    \begin{equation}
        \mu_\omega(\Omega_t)\ge (\beta t)^{\frac{2p}{3p-4}}\mu_\omega(\Omega_{t/2})^{\frac{p}{3p-4}},\,\,\, \beta=\frac{c^*}{16s^* M^{p-1}}.\nonumber
    \end{equation}
    By iteration, for any $m\in \mathbb{N}_+$
    \[\mu_\omega(\Omega_t)\ge \beta^{\sum_{j=0}^{m-1}2(\frac{p}{3p-4})^{j+1}}\left(\frac{1}{2}\right)^{\sum_{j=0}^{m-1}2j\left(\frac{p}{3p-4}\right)^{j+1}}t^{\sum_{j=0}^{m-1}2\left(\frac{p}{3p-4}\right)^{j+1}}\mu_\omega(\Omega_{t/2^m})^{\left(\frac{p}{3p-4}\right)^m}.\]
    It is elementary that 
    \[\sum_{j=0}^{\infty}\left(\frac{p}{3p-4}\right)^{j+1}=\frac{p}{2(p-2)}\]
    and $\sum_{j=0}^{m-1}2j\left(\frac{p}{3p-4}\right)^{j+1}$ converges.
    It suffices to show that
    \begin{equation}
        \lim_{m\to\infty} \mu_\omega(\Omega_{t/2^m})^{\left(\frac{p}{3p-4}\right)^m}=1.\nonumber
    \end{equation}
    By choosing $z_0$ such that $v(z_0)=M$ and setting $a=\Vert |\nabla v|_\omega\Vert_{L^\infty}$, we have $B_\omega(z_0,\frac{t}{2^m a})\subset \Omega_{t/2^m}$. For sufficiently large $m$, we have 
    \[\mu_\omega(\Omega_{t/2^m})\ge\mu_\omega\left(B_\omega(z_0,\frac{t}{2^m a})\right)\ge C(\omega)\left(\frac{t}{2^m a}\right)^{2n}.\]
    Then
    \[\lim_{m\to\infty} \mu_\omega(\Omega_{t/2^m})^{\left(\frac{p}{3p-4}\right)^m}\ge \lim_{m\to\infty}\left(C(\omega)\frac{t}{2^m a}\right)^{2n \left(\frac{p}{3p-4}\right)^m}=1.\]
    Therefore \refe{level set of v} holds.
    
  \subsection*{Step 4: Uniform positive lower bound of $s^*$: the special case that $M$ is uniformly bounded} 
We first consider the special case that $M:=\Vert v\Vert_{L^\infty}\le C$ for a uniform number $C=C(p,L)$. Recall a basic fact that for $p>2$, an $L^{\frac{p}{2},\infty}$ estimate implies an $L^1$ estimate (see \cite[page 14, Exercise 1.1.11]{Gra}). Precisely, in our setting assuming \eqref{weak-int}, there exists $C=C(p,L)$ such that for any $x\in X$,
    \begin{equation}
        \int_X |G_\omega(x,\cdot)|\omega^n=\int_X |G_\omega(x,\cdot)V_\omega|\frac{\omega^n}{V_\omega}\le C.\nonumber
    \end{equation}
    Since $v$ satisfies \refe{v}, we utilize Green's formula and the equivalence \eqref{s-lambda} to see 
    \begin{align*}
        M=|v(z_0)|\le \int_X |G_\omega(z_0,\cdot)||\Delta_\omega v|\omega^n\le C\Vert \Delta_\omega v\Vert_{L^\infty}\le C s^* M^{p-1},
    \end{align*}
    from which we conclude 
    $$s^*\ge c M^{2-p}\ge c.$$

  \subsection*{Step 5: Uniform positive lower bound of $s^*$: the general case} 
  Basing on \refe{level set of v}, we now prove that $s^*$ is uniformly bounded from below in general case.
  
    As a preparation, we first prove that there exists $C_0=C_0(p,L)>0$ such that $M\le (1+C_0)T$. If $M\le T$, then no further discussion is needed. If $T<M$, then let $V=\{v>T\}$. By the normalizing condition, we have 
    \begin{equation}\label{V}
        1=\int_X \mathcal{F}(v)\frac{\omega^n}{V_\omega}\ge \int_V \mathcal{F}(v)\frac{\omega^n}{V_\omega}\ge \frac{T^p}{p}\mu_\omega(V).
    \end{equation}
    Without loss of generality, we assume $T^p>2p$. Then $\mu_\omega(V)<1/2$.
    We consider the following auxiliary Possion equation
    \begin{equation}
        \begin{cases}
            \Delta_\omega u=\chi_{V}(-\hat{\lambda}f(v)+\hat{\lambda}\int_X f(v)\frac{\omega^n}{V_\omega})+\chi_{X\backslash V}\frac{\mu_\omega(V)}{\mu_\omega(X\backslash V)}\hat{\lambda}(T^{p-1}-\int_X f(v)\frac{\omega^n}{V_\omega}),\\
            \int_X u\omega^n=0.\nonumber
        \end{cases}
    \end{equation} 
    By Green's formula, \refe{integral in U} and  \refe{V}, for any $z\in X$
    \begin{align*}
        |u(z)|&\le \int_X |G_\omega(z,\cdot)||\Delta_\omega u|\omega^n\\
        &\le T^{p-1} \int_V |G_\omega(z,\cdot)|\omega^n+2T^{p-1}\frac{\mu_\omega(V)}{\mu_\omega(X\backslash V)}\int_{X\backslash V}|G_\omega(z,\cdot)|\omega^n\\
        &\le CT^{p-1}\mu_\omega(V)^{1-\frac{2}{p}}+C\mu_\omega(V)T^{p-1}\le C T.
    \end{align*}
    Note that $\Delta_\omega (v-u)=0$ in $V$ and $v=T$ on $\partial V$. By maximum principle,
    \[M=v(z_0)\le T+ 2\Vert u\Vert_{L^{\infty}}\le (2C+1)T=(C_0+1)T.\]
    As another preparation, we observe that by the normalizing condition $\int_X \mathcal{F}(v)e^F\omega_X^n=1,$ we have 
    \begin{align*}
        1&=\int_0^{\mathcal{F}(M)}\mu_\omega(\{z\in X: \mathcal{F}(v)>t\})\diff t\\
        &\ge \int_0^M\mu_\omega(\{z\in X: v>t\})\mathcal{F}'(t)\diff t\\
        &=\int_{0}^{M}\mu_\omega(\{z\in X: v>M-t\})\mathcal{F}'(M-t)\diff t\\
        &=\int_{0}^{M}\mu_\omega(\Omega_t)\mathcal{F}'(M-t)\diff t\\
        &\ge \int_{0}^{M}\min\left\{\frac{1}{2},C(p)\left(\frac{c^* t}{s^* M^{p-1}}\right)^{\frac{p}{p-2}}\right\}\mathcal{F}'(M-t)\diff t.
    \end{align*}
    We fix $t_0$ such that 
    \[\frac{1}{2}=C(p)\left(\frac{c^* t_0}{s^* M^{p-1}}\right)^{\frac{p}{p-2}}.\]
    
     In the remaining part, we discuss case by case.\\
    \subsubsection*{Case 1} $T>M, t_0>M$. We have 
    \begin{align*}
        1&\ge C(p)\left(\frac{c^*}{s^* M^{p-1}}\right)^{\frac{p}{p-2}}\int_0^M t^{\frac{p}{p-2}}(M-t)^{p-1}\diff t\\
        &=C(p)\left(\frac{c^*}{s^*}\right)^{\frac{p}{p-2}}\int_0^1 t^{\frac{p}{p-2}}(1-t)^{p-1}\diff t.
    \end{align*}
    Then 
    \begin{equation}\label{sgec}
    s^*\ge c\cdot c^*,
    \end{equation}
    which implies a positive lower bound for $s^*$, thanks to Faber-Krahn type inequality \eqref{c} established in Step 2.
    
     \subsubsection*{Case 2} $T>M, t_0\le M$. We have 
    \begin{equation}
        1\ge C(p)\left(\frac{c^*}{s^*}\right)^{\frac{p}{p-2}}\int_{0}^{t_0/M} t^{\frac{p}{p-2}}(1-t)^{p-1}\diff t+ \frac{1}{2}M^p\int_{t_0/M}^{1} (1-t)^{p-1}\diff t.\nonumber
    \end{equation}
    If $t_0/M>1/2$, we have 
    \[\int_{0}^{t_0/M} t^{\frac{p}{p-2}}(1-t)^{p-1}\diff t\ge\int_{0}^{1/2} t^{\frac{p}{p-2}}(1-t)^{p-1}\diff t=C.\]
    Then the inequality \refe{sgec} holds.\\
    If $t_0/M\le 1/2$, we have 
    \[\int_{t_0/M}^{1} (1-t)^{p-1}\diff t\ge\int_{1/2}^{1} (1-t)^{p-1}\diff t=C.\]
    Then $M\le C$, and we conclude a positive lower bound for $s^*$ by Step 4.
    
 \subsubsection*{Case 3} $T\le M, t_0> M$. In this case we have
    \begin{align*}
        1&\ge C(p)\left(\frac{c^*}{s^* M^{p-1}}\right)^{\frac{p}{p-2}}\int_0^M t^{\frac{p}{p-2}}f(M-t)\diff t\\
        &=C(p)\left(\frac{c^*}{s^*}\right)^{\frac{p}{p-2}}\int_0^1 t^{\frac{p}{p-2}}\frac{f(M-Mt)}{M^{p-1}}\diff t\\
        &\ge C(p)\left(\frac{c^*}{s^*}\right)^{\frac{p}{p-2}}\int_{1-T/M}^1 t^{\frac{p}{p-2}}(1-t)^{p-1}\diff t\\
        &\ge C(p)\left(\frac{c^*}{s^*}\right)^{\frac{p}{p-2}}\int_{1-\frac{1}{C_0+1}}^1 t^{\frac{p}{p-2}}(1-t)^{p-1}\diff t.
    \end{align*}
    Then the inequality \refe{sgec} holds.
    
    \subsubsection*{Case 4} $T\le M, t_0\le M$. If $M-T\ge t_0$, we have
    \begin{align*}
        1&\ge C(p)\left(\frac{c^*}{s^* M^{p-1}}\right)^{\frac{p}{p-2}}T^{p-1}\int_0^{t_0}t^{\frac{p}{p-2}} \diff t+\frac{1}{2}\int_{t_0}^{M-T}T^{p-1}\diff t+\frac{1}{2}\int_{M-T}^{M}(M-t)^{p-1}\diff t\\
        &\ge \frac{1}{2}\int_{M-T}^{M}(M-t)^{p-1}\diff t \ge\frac{M^p}{2}\int_{1-\frac{T}{M}}^{1}(1-t)^{p-1}\diff t \ge \frac{M^p}{2}\int_{1-\frac{1}{C_0+1}}^{1}(1-t)^{p-1}\diff t.
    \end{align*}
    Then $M\le C$, and we conclude a positive lower bound for $s^*$ by Step 4. 
    
    If $M-T< t_0$, we have 
    \begin{align*}
        1&\ge C(p)\left(\frac{c^*}{s^* M^{p-1}}\right)^{\frac{p}{p-2}}T^{p-1}\int_0^{M-T}t^{\frac{p}{p-2}} \diff t+C(p)\left(\frac{c^*}{s^* M^{p-1}}\right)^{\frac{p}{p-2}}\int_{M-T}^{t_0}t^{\frac{p}{p-2}}(M-t)^{p-1}\diff t\\
        &+\frac{1}{2}\int_{t_0}^{M}(M-t)^{p-1}\diff t\\
        &\ge C(p)\left(\frac{c^*}{s^* M^{p-1}}\right)^{\frac{p}{p-2}}\int_{M-T}^{t_0}t^{\frac{p}{p-2}}(M-t)^{p-1}\diff t+\frac{1}{2}\int_{t_0}^{M}(M-t)^{p-1}\diff t\\
        &= C(p)\left(\frac{c^*}{s^*}\right)^{\frac{p}{p-2}}\int_{1-T/M}^{t_0/M}t^{\frac{p}{p-2}}(1-t)^{p-1}\diff t+\frac{M^p}{2}\int_{t_0/M}^{1}(1-t)^{p-1}\diff t\\
        &\ge C(p)\left(\frac{c^*}{s^*}\right)^{\frac{p}{p-2}}\int_{1-\frac{1}{C_0+1}}^{t_0/M}t^{\frac{p}{p-2}}(1-t)^{p-1}\diff t+\frac{M^p}{2}\int_{t_0/M}^{1}(1-t)^{p-1}\diff t.\\
    \end{align*}
    We discuss two different cases. If $t_0/M\le 1-\frac{1}{2C_0+1}$, we have  
    \begin{equation*}
        1\ge\frac{M^p}{2}\int_{t_0/M}^{1}(1-t)^{p-1}\diff t \ge\frac{M^p}{2}\int_{\frac{1}{2C_0+1}}^{1}(1-t)^{p-1}\diff t.
    \end{equation*}
    Then $M\le C$ and we conclude a positive lower bound for $s^*$ by Step 4. If $t_0/M> 1-\frac{1}{2C_0+1}$, we have 
    \begin{equation*}
        1\ge C(p)\left(\frac{c^*}{s^*}\right)^{\frac{p}{p-2}}\int_{1-\frac{1}{C_0+1}}^{t_0/M}t^{\frac{p}{p-2}}(1-t)^{p-1}\diff t \ge C(p)\left(\frac{c^*}{s^*}\right)^{\frac{p}{p-2}}\int_{1-\frac{1}{C_0+1}}^{1-\frac{1}{2C_0+1}}t^{\frac{p}{p-2}}(1-t)^{p-1}\diff t,
    \end{equation*}
    and hence the inequality \refe{sgec} holds.
    
    In conclusion, we have bounded $s^*$ from below by a uniform positive number, and the proof of $(1)\Rightarrow (2)$ is completed.

    The followings are two lemmas used above.
    \begin{lem}\label{existence and regularity of v}
        There is a function $v\in C^{2,\alpha}(X)$ such that \refe{s*} attains infimum.
    \end{lem}
    \begin{proof}
        Let $\{v_n\}\subset H^1(X)$ be a minimizing sequence satisfying $\int_X \mathcal{F}(v_n)e^F\omega_X^n=1, \int_X v_n\omega^n=0$ and $\int_X |\nabla v_n|^2 e^F\omega_X^n\to s^* (n\to\infty)$. Then by Poincar\'e inequality, $\{v_n\}$ is bounded in $H^1(X)$. Since $H^1(X)$ compactly embeds into $L^1(X)$, after passing to a subsequence, we have 
        \begin{equation}
            v_n\rightharpoonup v \,\,\mathrm{in}\,\, H^1(X), v_n\to v \,\, \mathrm{in}\,\, L^1(X).\nonumber
        \end{equation}
        The function $\mathcal{F}$ is Lipschitz(its derivative $f$ is bounded), therefore $\mathcal{F}(v_n)\to \mathcal{F}(v)$ in $L^1(X)$. Consequently
        \[\int_X\mathcal{F}(v)e^F\omega_X^n=1, \int_X v e^F\omega_X^n=0.\]
        Thus the limit $v$ satisfies the constraints. By the weak lower semicontinuity in $H^1(X)$, we have 
        \[\int_X |\nabla v|^2_\omega e^F\omega_X^n\le \liminf_{n\to\infty} \int_X |\nabla v_n|^2_\omega e^F\omega_X^n=s^*.\]
        Hence $v$ attains $s^*$. From the previous discussion,  $v$ satisfies the following equation
        \begin{equation*}
            \begin{cases}
                \Delta_\omega v=g\\
                \int_X v e^F\omega_X^n=0
            \end{cases}
        \end{equation*}
        where $g=-\hat{\lambda}f(v)+\hat{\lambda}\int_X f(v)e^F\omega_X^n\in L^{\infty}(X)$. By elliptic $L^p$ regularity and a standard bootstrapping argument, we obtain $v\in C^{1,\alpha}(X)$. Since $f$ is Lipschitz, $g$ becomes H\"older continuous. Now Schauder estimates yield that $v\in C^{2,\alpha}(X)$.
    \end{proof}

    \begin{lem}\label{existence and regularity of psi}
         For any subset $U\subset X$, there is a function $\phi \in C^{1,\alpha}(X)$ such that \refe{lambda(U)} attains infimum.
    \end{lem}
    \begin{proof}
        Using the compact embedding  $H^1(X) \hookrightarrow L^2(X)$, we can prove $\phi\in H^1(X)$ by the same argument as \refl{existence and regularity of v}. From the previous discussion,  $\phi$ satisfies 
         \begin{equation}\label{psi h}
        \begin{cases}
            \Delta_\omega \phi=h,\\
            \int_X\phi\omega^n=0,\nonumber
        \end{cases}
    \end{equation}
    where $h=-\lambda(U)\phi\chi_U+\lambda(U)\int_U \phi e^F\omega_X^n\in L^2(X)$. By elliptic $L^p$ regularity and a standard bootstrapping argument, we obtain $\phi,h\in L^\infty(X)$. Furthermore, using equation concludes $\phi\in C^{1,\alpha}(X)$.
    \end{proof}

\subsection{Proof of \reft{thm-green-sob}, Part Two: $(2)\Rightarrow (1)$.} 
   Next we prove the implication $(2)\Rightarrow(1)$. Assume $(2)$ holds.
   
  \subsection*{Step I: $L^\infty$ estimate for Poisson equation.} We need the following lemma.
    \begin{lem}\label{infty est of Poi}
        Let $u\in H^1(X)$ be the weak solution to 
        \begin{equation}
            \begin{cases}
                -\Delta_\omega u=f\\
                \int_X u\omega^n=0
            \end{cases}
        \end{equation}
        where $\int_X f\omega^n=0$. There exists constant $C>0$ that only depends on $p, C_S$ such that 
        \begin{equation}
            \Vert u\Vert_{L^\infty(X)}\le C\Vert f\Vert_{L^\infty}.
        \end{equation}
    \end{lem}
    \begin{proof}
This follows by applying the Sobolev-type inequality \eqref{Sob ineq} and Moser iteration. For completeness we include a proof here. Without loss of generality, we assume $\Vert f\Vert_{L^\infty(X)}=1$. By taking the test function $u$, we have 
        \[
            \frac{1}{V_\omega}\int_X fu\omega^n=-\frac{1}{V_\omega}\int_X u\Delta_\omega u\omega^n=\frac{1}{V_\omega}\int_X |\nabla u|^2_\omega \omega^n.
        \]
        Using the assumption $\Vert f\Vert_{L^\infty(X)}=1$ and H\"older's inequality, we have 
        \[
            \frac{1}{V_\omega}\int_X fu\omega^n\le\frac{1}{V_\omega}\int_X |u|\omega^n\le \Vert u\Vert_{L^p(X,\frac{\omega^n}{V_\omega})}.
        \]
        By Sobolev's inequality \refe{Sob ineq}, we have 
        \[
            \frac{1}{V_\omega}\int_X |\nabla u|^2_\omega \omega^n\ge C_S^{-2}\Vert u\Vert^2_{L^p(X,\frac{\omega^n}{V_\omega})}.
        \]
        Thus we have 
        \begin{equation}\label{u Lp}
            \Vert u\Vert_{L^p(X,\frac{\omega^n}{V_\omega})}\le C.
        \end{equation}
        Next by taking the test function $|u|^{r-2}u$ for any $r>2$, we have 
        \begin{equation*}
            -\frac{1}{V_\omega}\int_X |u|^{r-2}u\Delta_\omega u\omega^n=\frac{1}{V_\omega}\int_X f|u|^{r-2}u\omega^n
        \end{equation*}
        By Stokes' formula, the left-hand side satisfies
        \begin{equation*}
            -\frac{1}{V_\omega}\int_X |u|^{r-2}u\Delta_\omega u\omega^n=\frac{4(r-1)}{r^2}\int_X |\nabla (u^{r/2})|^2_\omega \frac{\omega^n}{V_\omega}.
        \end{equation*}
        Combining the assumption $\Vert f\Vert_{L^\infty(X)}=1$, we have 
        \begin{equation}\label{grad u est}
            \int_X |\nabla (u^{r/2})|^2_\omega \frac{\omega^n}{V_\omega}\le \frac{r^2}{4(r-1)}\int_X u^{r-1}\frac{\omega^n}{V_\omega}.
        \end{equation}
        By Sobolev's inequality \refe{Sob ineq}
        \begin{align*}
            \left(\frac{1}{V_\omega}\int_X |u|^{\frac{rp}{2}}\omega^n\right)^{1/p}&\le \left(\frac{1}{V_\omega}\int_X |u|^{r/2}\omega^n\right)+C_S\left(\frac{1}{V_\omega}\int_X |\nabla(u^{r/2})|^2_\omega \omega^n\right)^{1/2}\\
            \text{by \refe{grad u est}} &\le \left(\frac{1}{V_\omega}\int_X |u|^{r/2}\omega^n\right)+Cr^{1/2}\left(\frac{1}{V_\omega}\int_X |u|^{r-1} \omega^n\right)^{1/2}\\
            \text{(by H\"older's ineq.)} &\le\left(\frac{1}{V_\omega}\int_X |u|^r\omega^n\right)^{1/2}+Cr^{1/2}\left(\frac{1}{V_\omega}\int_X |u|^r\omega^n\right)^{\frac{r-1}{2r}}.
        \end{align*}
        Squaring the inequalities gives
        \begin{align*}
            \left(\frac{1}{V_\omega}\int_X |u|^{\frac{rp}{2}}\omega^n\right)^{2/p}&\le (1+C)r (\frac{1}{V_\omega}\int_X |u|^r\omega^n+\left(\frac{1}{V_\omega}\int_X |u|^r\omega^n\right)^{1-\frac{1}{r}})\\
            &\le Cr\max\{1,\Vert u\Vert_{L^r(X,\frac{\omega^n}{V_\omega})}^r\}.
        \end{align*}
        Taking the r-th root gives
        \[\Vert u\Vert_{L^{\frac{rp}{2}}(X,\frac{\omega^n}{V_\omega})}\le (Cr)^{1/r}\max\{1,\Vert u\Vert_{L^r(X,\frac{\omega^n}{V_\omega})}\}.\]
        We define $r_0=p, r_{k+1}=\frac{p}{2}r_k$ and $M_k=\max\{1,\Vert u\Vert_{L^{r_k}(X,\frac{\omega^n}{V_\omega})}\}$, then by Moser's iteration
        \[
            M_{k+1}\le (Cr_k)^{1/r_k}M_k\le C^{\frac{1}{p}\sum_{i=0}^{k}\left(\frac{2}{p}\right)^i}(p/2)^{\frac{1}{p}\sum_{i=0}^{k}i\left(\frac{2}{p}\right)^i}M_0.
        \]
        Letting $k\to\infty$ and using \refe{u Lp}, we have 
        \[\Vert u\Vert_{L^\infty}\le C M_0=C \max\{1,\Vert u\Vert_{L^p(X,\frac{\omega^n}{V_\omega})}\}\le C.\]
    \end{proof}
    
    \subsection*{Step II: Bounding $L^1$-norm of Green's function.} 
    Next let's prove that $L^1$-norm of Green's function is uniformly bounded in terms of $C_S$ and $p$. For any fixed $x\in X$, we regard $G(y):=G_\omega(x,y)$ and consider the following linear equation
    \begin{equation}
        \begin{cases}
            -\Delta_\omega u=\chi_{\{G\ge0\}}-\frac{1}{V_\omega}\int_{\{G\ge 0\}} \omega^n=:f\\
            \int_X u\omega^n=0.
        \end{cases}
    \end{equation}
    Since $\Vert f\Vert_{L^\infty}\le 1$, by \refl{infty est of Poi} there exists a constant $C=C(C_S,p)>0$ such that $\Vert u\Vert_{L^\infty}\le C$.
    By Green's formula, we have 
    \[u(x)=\int_X G(x,y)f(y)\omega^n=\int_{\{G\ge 0\}} G\omega^n\le C.\]
    Combining this with the fact that 
    \[\int_{\{G\ge0\}}G\omega^n=-\int_{\{G\le 0\}}G\omega^n,\]
    we find that 
    \[\int_X |G|\omega^n\le C_1.\]
    
    \subsection*{Step III: Bounding $L^{\frac{p}{2},\infty}$-norm of Green's function.} Finally, we upgrade the $L^1$ bound to an $L^{\frac{p}{2},\infty}$ bound. For any $t>0$, we let 
    \begin{equation}
        u_t(y)=\begin{cases}
            t, &G(y)V_\omega>t;\\
            G(y)V_\omega, &|G(y)V_\omega|\le t;\\
            -t, &G(y)V_\omega<-t.
        \end{cases}
    \end{equation} 
    By Green's formula 
    \[t=u_t(x)=\frac{1}{V_\omega}\int_X u_t\omega^n+\int_X \langle\nabla G,\nabla u_t\rangle_\omega \omega^n\]
    we have 
    \begin{align*}
        t-\frac{1}{V_\omega}\int_X u_t\omega^n&=\frac{1}{V_\omega}\int_X |\nabla u_t|^2_\omega \omega^n\\
        \text{(by Sobolev's ineq. \refe{Sob ineq})} &\ge c \left(\frac{1}{V_\omega}\int_X|u_t-\bar{u_t}|^p\omega^n\right)^{2/p}\\
        &\ge c\left(\frac{1}{V_\omega}\int_{A_{t+}}|t-\bar{u_t}|^p\omega^n+\frac{1}{V_\omega}\int_{A_{t-}}|t+\bar{u_t}|^p\omega^n\right)^{2/p}\\
        &= c\left(\mu_\omega(A_{t+})|t-\bar{u_t}|^p+\mu_\omega(A_{t-})|t+\bar{u_t}|^p\right)^{2/p},
    \end{align*}
    where $A_{t+}=\{GV_\omega> t\}$, $A_{t-}=\{GV_\omega< -t\}$, $\overline{u_t}=\frac{1}{V_\omega}\int_X u_t\omega^n$ and the constant $c=C_S^{-2}$. Since 
    \[|\overline{u_t}|\le \frac{1}{V_\omega}\int_X |u_t|\omega^n\le \int_X |G|\omega^n\le C_1,\]
    we have for $t>2C_1$
    \[2t\ge t-\frac{1}{V_\omega}\int_X u_t\omega^n\ge c t^2\left(\mu_\omega(A_{t+})+\mu_\omega(A_{t-})\right)^{2/p}=c t^2\left(\mu_\omega(A_{t})\right)^{2/p},\]
    where $A_t=\{|GV_\omega|>t\}=A_{t-}\cup A_{t+}$. Thus we have $t^{p/2}\mu_\omega(A_t)\le C$, which completes the proof of $(1)\Rightarrow (2)$.
    
    In conclusion, \reft{thm-green-sob} is proved.

    \begin{rem}
    Using identical arguments, we have the same conclusion in \reft{thm-green-sob} for a closed Riemannian manifold.
    \end{rem}

    \section{$L^{\frac{nq}{nq-q+n},\infty}$ estimate on Green's function}\label{sect-green}
Main result in this section is an $L^{\frac{nq}{nq-q+n},\infty}$ estimate on Green's function under $q$-Nash entropy condition.
\begin{thm}\label{thm-green}
    Let $X$ be an $n$-dimensional compact K\"ahler manifold equipped with a K\"ahler metric $\omega_X$ of $V_{\omega_X}=1$, then for any $A,K>0$, $q>n$, there exists $C=C(X,\omega_X,n,A,q,K,p)>0$ such that for any $\omega\in\mathcal{V}(X,\omega_X,n,A,q,K)$ and $x\in X$, we have the following estimate
    \begin{equation}\label{weak q-norm}
        \Vert G(x,\cdot)V_\omega\Vert_{L^{p/2,\infty}(X,\omega)}:=\sup_{t>0}\{t^{p/2}\mu_\omega(\{y\in X:|G(x,y)V_\omega|>t\})\}\le C.
    \end{equation}
    where $p:=\frac{2nq}{nq-q+n}$, $\mu_\omega(U):=\int_U e^F\omega_X^n=\frac{1}{V_\omega}\int_U \omega^n$ for any subset $U\subset X$.
\end{thm}
The exponent $\frac{2nq}{nq-q+n}$ here is optimal, see Remark \ref{rem-opt}.

\subsection{Proof of Theorem \ref{thm-green}}
\subsection*{Step 1: $L^\infty$ estimate for complex complex Monge-Amp\`ere equations} We will empoy a version of $L^\infty$ estimate for complex Monge-Amp\`ere equation. Let's begin with a simple observation.
\begin{lem}\label{q-norm estimate}
   Given $q>0$. For any $\omega\in\mathcal{K}(X)$ and $F$ satisfying $\Vert e^F\Vert_{L^1\log L^q(X,\omega_X)}\le K$, if $\theta\in [\omega]$ satisfies $\theta\le \omega_X$, then there exists constant $C=C(X,\omega_X,n,q,K)$ such that for any $\varphi\in \mathrm{PSH}(X,\theta)$ satisfying $\sup_X \varphi=0$, we have
    \begin{equation*}
        \int_X (-\varphi)^q e^F\omega^n_X\le C.
    \end{equation*}
\end{lem}
\begin{proof}
    This is essentially contained in \cite[Subsection 2.2]{GL} using general theory on Orlicz space (when $q>n$). Here we present a direct simple proof. By the assumption on $\theta$, we have $\mathrm{PSH}(X,\theta)\subseteq\mathrm{PSH}(X,\omega_X)$. Applying the $\alpha$-invariant of $\omega_X$ (see \cite{T}), there exist constants $\alpha,\,C>0$ such that 
    \begin{equation*}
        \int_X e^{-\alpha \varphi}\omega_X^n\le C.
    \end{equation*}
    Define $\eta:\mathbb{R}_+\to\mathbb{R}_+$ by $\eta(x)=\frac{1}{\beta}\log^q(1+x)$ where $\beta>0$ is to be determined. By the generalized Young's inequality with respect to $\eta$ we see
    \begin{equation*}
        \begin{aligned}
            \int_X &(-\varphi)^q e^F\omega_X^n\leq \int_X (-\varphi)^q \eta^{-1}((-\varphi)^q)\omega_X^n+\int_X e^F\eta(e^F)\omega_X^n\\
            &=\int_X (-\varphi)^q(\exp\{\beta^{1/q}(-\varphi)\}-1)\omega_X^n+\beta\int_X\log^q(1+e^F)e^F\omega_X^n\\
            &\le C\int_X \exp\{2\beta^{1/q}(-\varphi)\}\omega_X^n+K\le C,\nonumber
        \end{aligned}
    \end{equation*}
    where we have chosen $\beta$ with $2\beta^{1/q}=\alpha$.
    
    Lemma \ref{q-norm estimate} is proved.
\end{proof}

\begin{lem}\label{infty-norm estimate}
    Under the same assumptions of \refl{q-norm estimate} with $q>n$, for any smooth positive function $h:X\to \mathbb{R_+}$, if $\varphi$ solves the following complex Monge-Amp\`ere equation 
    \begin{equation}
        \frac{1}{V_\omega}(\theta+\dd\varphi)^n=\frac{h}{A} e^F\omega_X^n,\,\,\, \sup_X\varphi=0,\nonumber
    \end{equation}
    where $A=\int_X h e^F\omega_X^n$, then there is a constant $C(X,\omega_X,n,q,K)>0$ such that
    \begin{equation}
        0\le -\varphi+ u_\theta\le C \left(\frac{\sup_X h}{A}\right)^{\frac{1}{q}}\,\,\,\text{on X},\nonumber
    \end{equation}
    where $u_\theta:=\sup\{u|u\in\mathrm{PSH}(X,\theta),u\le 0\}$.
\end{lem}
The proof makes use of method of auxiliary complex Monge-Amp\`ere equation in \cite{GPT,GPTW}.
\begin{proof}
    Firstly we may assume $\sup_X h=1$. Thanks to Berman's result \cite[Theorem1.1]{B}, we fix a sequence of smooth function $u_\beta\in\mathrm{PSH}(X,\theta)$,$\beta\in\mathbb{N}$, that $u_\beta$ converges uniformly to $ u_\theta$ as $\beta\to\infty$. We also choose a sequence of positive functions $\tau_k:\mathbb{R}\to\mathbb{R}_+$ that $\tau_k(x)$ decreases to $x\cdot\chi_{\mathbb{R}_+}(x)$ as $k\to\infty$. By Yau's theorem \cite{Y}, for any $s>0$, we solve the following auxiliary Monge-Amp\`ere equation on $X$
    \begin{equation*}
        (\theta+\dd\psi_{s,k,\beta})^n=\frac{\tau_k(-\varphi+u_\beta-s)}{B_{s,k,\beta}}(\theta+\dd\varphi)^n, \sup_X \psi_{s,k,\beta}=0,
    \end{equation*}
    where
    \[B_{s,k,\beta}=\frac{1}{V_\omega}\int_X\tau_k(-\varphi+u_\beta-s)(\theta+\dd\varphi)^n=\frac{1}{A}\int_X\tau_k(-\varphi+u_\beta-s)h e^F\omega_X^n.\]
    Since $\psi_{s,k,\beta}\le u_\theta$ and $u_\beta$ converges uniformly to $ u_\theta$, by taking $\beta$ large enough, we may assume $\psi_{s,k,\beta}\le u_\beta+1$.\\
    Define a function 
    \[\Phi=-\varepsilon(-\psi_{s,k,\beta}+u_\beta+1+\Lambda)^{\frac{n}{n+1}}-(\varphi-u_\beta+s)\]
    with the constants
    \begin{equation*}
        \varepsilon^{n+1}=B_{s,k,\beta}n^{-n}(n+1)^n, \Lambda=n^{n+1}(n+1)^{-n-1}\varepsilon^{n+1}.
    \end{equation*}
    As a smooth function on the compact manifold $X$, $\Phi$ must achieve its maximum at some point $x_0\in X$. We calculate ($\Delta_{\varphi}$ denotes the Laplacian with respect to the metric $\omega_\varphi=\theta+\dd\varphi$)
    \begin{align*}
        0\ge&\Delta_\varphi\Phi(x_0)\\
        =&-\frac{n\varepsilon}{n+1}(-\psi_{s,k,\beta}+u_\beta+1+\Lambda)^{-\frac{1}{n+1}}\mathrm{tr}_{\omega_\varphi}(-\dd\psi_{s,k,\beta}+\dd u_\beta)\\
        &-\mathrm{tr}_{\omega_\varphi}(\dd\varphi-\dd u_\beta)\\
        &+\frac{n\varepsilon}{(n+1)^2}(-\psi_{s,k,\beta}+u_\beta+1+\Lambda)^{-\frac{n+2}{n+1}}\mathrm{tr}_{\omega_\varphi}\sqrt{-1}\partial(\psi_{s,k,\beta}-u_\beta)\wedge\bar{\partial}(\psi_{s,k,\beta}-u_\beta)\\
        \ge&\frac{n\varepsilon}{n+1}(-\psi_{s,k,\beta}+u_\beta+1+\Lambda)^{-\frac{1}{n+1}}\mathrm{tr}_{\omega_\varphi}((\theta+\dd\psi_{s,k,\beta})-(\theta+\dd u_\beta))-n\\
        &+\mathrm{tr}_{\omega_\varphi}(\theta+\dd u_\beta)\\
        \ge&\frac{n\varepsilon}{n+1}(-\psi_{s,k,\beta}+u_\beta+1+\Lambda)^{-\frac{1}{n+1}}n\left(\frac{(\theta+\dd\psi_{s,k,\beta})^n}{\omega_\varphi^n}\right)^{1/n}-n\\
        &+(1-\frac{n\varepsilon}{n+1}(-\psi_{s,k,\beta}+u_\beta+1+\Lambda)^{-\frac{1}{n+1}})\mathrm{tr}_{\omega_\varphi}(\theta+\dd u_\beta)\\
        \ge&\frac{n^2\varepsilon}{n+1}(-\psi_{s,k,\beta}+u_\beta+1+\Lambda)^{-\frac{1}{n+1}}(\tau_k(-\varphi+u_\beta-s)B^{-1}_{s,k,\beta})^{1/n}-n\\
        &+(1-\frac{n\varepsilon}{n+1}\Lambda^{-\frac{1}{n+1}})\mathrm{tr}_{\omega_\varphi}(\theta+\dd u_\beta)\\
        \ge&\frac{n^2\varepsilon}{n+1}(-\psi_{s,k,\beta}+u_\beta+1+\Lambda)^{-\frac{1}{n+1}}(\tau_k(-\varphi+u_\beta-s)B^{-1}_{s,k,\beta})^{1/n}-n
    \end{align*}
    Therefore, at $x_0$,
    \begin{align*}
        -\varphi+u_\beta-s&\le\tau_k(-\varphi+u_\beta-s)\\
        &\le\left(\frac{n+1}{n\varepsilon}\right)^n B_{s,k,\beta}(-\psi_{s,k,\beta}+u_\beta+1+\Lambda)^{\frac{n}{n+1}}\\
        &=\varepsilon(-\psi_{s,k,\beta}+u_\beta+1+\Lambda)^{\frac{n}{n+1}},
    \end{align*}
    i.e. $\Phi(x_0)\le0$. We have
    \[-\varphi+ u_\theta-s\le-\varphi+u_\beta-s+\epsilon_\beta\le\varepsilon(-\psi_{s,k,\beta}+u_\beta+1+\Lambda)^{\frac{n}{n+1}}+\epsilon_\beta\]
    where $\epsilon_\beta\to0$ as $\beta\to\infty$.\\
    On $\Omega_s:=\{-\varphi+ u_\theta-s>0\}$ we have
    \begin{align*}
        (-\varphi+ u_\theta-s)^{\frac{n+1}{n}}\le C_n B_{s,k,\beta}^{1/n}(-\psi_{s,k,\beta}+u_\beta+1+B_{s,k,\beta})+\epsilon^{\frac{n+1}{n}}_\beta.
    \end{align*}
    Since by definition $ u_\theta\le0$, we have $u_\beta\le\epsilon_\beta$. Applying \refl{q-norm estimate} to $\psi_{s,k,\beta}$, we obtain that there exists $C=C(X,\omega_X,n,q,K)$ such that 
    \begin{align*}
        \int_{\Omega_s}\left(\frac{(-\varphi+ u_\theta-s)^{1+1/n}}{B_{s,k,\beta}^{1/n}}\right)^q he^F\omega_X^n&\le C\left(\int_{\Omega_s}(-\psi_{s,k,\beta})^q he^F\omega_X^n+\epsilon_\beta^q+B_{s,k,\beta}^q \mu_h(\Omega_s)+\frac{\epsilon_\beta^{q+q/n}}{B^{q/n}_{s,k,\beta}}\right)\\
        &\le C\left(1+B_{s,k,\beta}^q\mu_h(\Omega_s)+\frac{\epsilon_\beta^{q+q/n}}{B^{q/n}_{s,k,\beta}}\right),
    \end{align*}
    where $\mu_h(\Omega_s)=\int_{\Omega_s}he^F\omega_X^n$. 
    
    Letting firstly $\beta\to\infty$ and then $k\to\infty$, we have 
    \begin{equation}\label{q-norm with Bs}
        \int_{\Omega_s}\left(\frac{(-\varphi+ u_\theta-s)^{1+1/n}}{B_{s}^{1/n}}\right)^q he^F\omega_X^n \le C(1+B_s^q\mu_h(\Omega_s)).
    \end{equation}
    where $B_s=\frac{1}{V_\omega}\int_{\Omega_s}(-\varphi+ u_\theta-s)(\theta+\dd\varphi)^n=\frac{1}{A}\int_{\Omega_s}(-\varphi+ u_\theta-s)h e^F\omega_X^n$. The second part of \refe{q-norm with Bs} is bounded by H\"older inequality:
    \begin{align*}
        B_s&=\frac{1}{A}\int_{\Omega_s}(-\varphi+ u_\theta-s)he^F\omega_X^n=\frac{\mu_h(\Omega_s)}{A}\frac{1}{\mu_(\Omega_s)}\int_{\Omega_s}(-\varphi+ u_\theta-s)he^F\omega_X^n\\
        &\le \left(\frac{1}{\mu_h(\Omega_s)}\int_{\Omega_s}(-\varphi)^q he^F\omega_X^n\right)^{1/q}\le C\frac{1}{(\mu_h(\Omega_s))^{1/q}}.
    \end{align*} 
    From the definition of $B_s$, it follows from H\"older inequality that
    \begin{align*}
        B_s^{\frac{n}{n+1}}&=\frac{1}{A}\int_{\Omega_s}\frac{(-\varphi+ u_\theta-s)}{B_s^{\frac{1}{n+1}}}he^F\omega_X^n\\
        &\le\frac{\mu_h(\Omega_s)}{A}\left(\frac{1}{\mu_h(\Omega_s)}\int_{\Omega_s}\frac{(-\varphi+ u_\theta-s)^{q(1+1/n)}}{B_s^{q/n}}he^F\omega_X^n\right)^{\frac{n}{q(n+1)}}\\
        &\le C\frac{1}{A}\mu_h(\Omega_s)^{1-\frac{n}{q(n+1)}}.
    \end{align*}
    We then get 
    \begin{equation}\label{iteration}
        r\mu_h(\Omega_{s+r})\le A\cdot B_s\le C_0\frac{1}{A^{1/n}}\mu_h(\Omega_s)^{1-\frac{1}{n}-\frac{1}{q}}.
    \end{equation}
    By a classic iteration of De Giorgi, we choose sequences $\{s_i\}_{i=0}^{\infty}, \{r_i\}_{i=0}^{\infty}$ and set $\Omega_i:=\Omega_{s_i}$ satisfying
    \[s_0=r_0=0, s_i=s_{i-1}+r_i, r_i=2C_0\frac{1}{A^{1/n}}\mu_h(\Omega_{i-1})^{1/n-1/q}\] 
    for $i\ge 1$. According to \refe{iteration}, $\mu_h(\Omega_i)\le\frac{1}{2}\mu_h(\Omega_{i-1})$. Let $s_\infty=\lim_{i\to\infty}s_i$, then 
    $\mu_h(\Omega_\infty)=0$ and 
    \begin{align*}
        s_\infty&=\sum_{i=1}^{\infty}r_i=\frac{2C_0}{A^{1/n}}\sum_{i=1}^{\infty}\mu_h(\Omega_{i-1})^{1/n-1/q}\\
        &\le\frac{2C_0}{A^{1/n}}\sum_{i=1}^{\infty}(\frac{1}{2})^{(i-1)(1/n-1/q)}\mu_h(\Omega_0)^{1/n-1/q}\le CA^{-1/q}.
    \end{align*}
    We have completed the proof of Lemma \ref{infty-norm estimate}.
\end{proof}

\subsection*{Step 2: Completion of proof of \reft{thm-green}}
Thanks to \cite{GPSS1} and \cite{GPSS2}, there exists $C=C(X,\omega_X,n,A,q,K)$ such that for any $x\in X$,
    \begin{equation}\label{inf and L1 estimate}
        -V_\omega\inf_{y\in X} G_\omega(x,y)+\int_X|G_\omega(x,\cdot)|\omega^n\le C.
    \end{equation}
    We write 
    \begin{equation}
        \mathcal{G}(x,\cdot)=G_\omega(x,\cdot)-\inf_{x,y\in X}G_\omega(x,y)+V_\omega^{-1}>0.
    \end{equation}
    By \refe{inf and L1 estimate} we have $\int_X \mathcal{G}(x,\cdot)\omega^n\le C$.
    We define $K_t=\{\mathcal{G}(x,\cdot)V_\omega>t\}$ for any $t>0$, then \refe{weak q-norm} is equivalent to 
    \begin{equation}
        \sup_{t>0}\{t^p\mu_\omega(K_t)\}\le C.
    \end{equation}
    Let $\{h_k\}_{k=1}^{\infty}$ be a sequence of smooth positive functions such that $\sup_X h_k\le 2$ and 
    \[\int_X |h_k-\chi_{K_t}|e^F\omega_X^n\to 0\]
    as $k\to\infty$. Clearly there holds
    \[A_k:=\int_X h_k e^F\omega_X^n\to \mu_\omega(K_t)\] as $k\to\infty$.
    According to \cite[Proposition 3.1]{GPSS1}, we can pick a smooth representative $\theta\in[\omega]$ such that $\Vert\theta\Vert_{L^\infty(X,\omega_X)}$ is uniformly bounded by some constant that only depends on $A$ and $\omega_X$. We consider the following auxiliary Monge-Amp\`ere equation on X
    \begin{equation*}
        \frac{1}{V_\omega}(\theta+\dd\varphi_k)^n=\frac{h_k}{A_k}e^F\omega_X^n, \,\,\,\sup_X\varphi_k=0.
    \end{equation*}
    We denote $\omega=\theta+\dd u$ with $\sup_X u=0$ and consider the function $v_k=\varphi_k-u$. It follows from \refl{infty-norm estimate} and $L^{\infty}$-estimate from \cite[Proposition 4.1]{GPSS1} that 
    \begin{equation}\label{infty-norm of vk}
        -CA_k^{-1/q}\le \varphi_k-u_\theta\le v_k\le u_\theta-u\le C.
    \end{equation}
    We calculate
    \begin{align*}
        \Delta_\omega v_k&=\mathrm{tr}_\omega(\theta+\dd\varphi_k)-n\\
        &\ge\left(\frac{(\theta+\dd\varphi_k)^n}{\omega^n}\right)^{1/n}-n\\
        &=\frac{h_k^{1/n}}{A_k^{1/n}}-n
    \end{align*}
    and then apply the Green's formula to $v_k$ at $x\in X$ to see
    \begin{align*}
        v_k(x)&=\frac{1}{V_\omega}\int_X v_k\omega^n+\int_X\mathcal{G}(x,\cdot)(-\Delta_\omega v_k)\omega^n\\
        &\le\frac{1}{V_\omega}\int_X v_k\omega^n-\frac{1}{A_k^{1/n}}\int_X \mathcal{G}(x,\cdot)h_k^{1/n}\omega^n+n\int_X\mathcal{G}(x,\cdot)\omega^n\\
        &\le C-\frac{1}{A_k^{1/n}}\int_X \mathcal{G}(x,\cdot)h_k^{1/n}\omega^n
    \end{align*}
    Combining \refe{infty-norm of vk} and the fact that $A_k\le 2$, we have 
    \begin{equation}\label{Ak and G}
        A_k^{1/q-1/n}\int_X\mathcal{G}(x,\cdot)h_k^{1/n}\omega^n\le C.
    \end{equation}
    Since $\mathcal{G}(x,\cdot)$ is integrable and $h_k$ is bounded, applying the Dominated Convergence Theorem, we have
    $\int_X\mathcal{G}(x,\cdot)h_k^{1/n}\omega^n\to \int_{K_t}\mathcal{G}(x,\cdot)\omega^n$
    as $k\to\infty$. By the definition of $K_t$, we have
    $\int_{K_t}\mathcal{G}(x,\cdot)\omega^n\ge t\mu_\omega(K_t).$
    Letting $k\to\infty$ in \refe{Ak and G} concludes that 
    \[t\mu_\omega(K_t)^{\frac{nq-q+n}{nq}}\le C,\]
    where $C=C(X,\omega_X,n,A,q,K)$.
    
    Theorem \ref{thm-green} is proved.

    \begin{rem}[Optimality of geometric estimates]\label{rem-opt} We now explain that the Sobolev-type inequality, local volume non-collapsing, weak integrability of Green's function and on-diagonal upper bound on heat kernel obtained in this paper (see Theorems \ref{Sobolev inequality} and \ref{thm-green} and Corollary \ref{cor}) are optimal under $q$-Nash entropy condition.
    \begin{itemize}
    \item[(1)] Firstly, the local volume non-collapsing estimate in Theorem \ref{Sobolev inequality} (2) is optimal in the sense that for any (small) $\epsilon>0$, the corresponding estimate with exponent of $R$ replaced by $\frac{2nq}{q-n}-\epsilon$ is no longer true, thanks to \cite[Example 3.1]{GS}.
    \item[(2)] Given the above item (1), the exponent $p=\frac{2nq}{nq-q+n}$ in Sobolev inequality \eqref {optimal-sob} is also optimal in the sense that for any $\epsilon>0$, the corresponding Sobolev inequality with exponent $p=\frac{2nq}{nq-q+n}+\epsilon$ is no longer true, since, on a closed Riemannian manifold, a Sobolev inequality with exponent $p>2$ implies local volume non-collapsing with exponent $\frac{2p}{p-2}$ (see e.g. \cite{GPSS3,H,Z}). Elementarily note that for $p>\frac{2nq}{nq-q+n}$, we have $\frac{2p}{p-2}<\frac{2nq}{q-n}$.
    \item[(3)] For Green's function estuimate, we observe that the conclusion in Theorem \ref{thm-green} is also optimal as a uniform $L^{\frac{nq}{nq-q+n}+\epsilon,\infty}$ estimate for Green' function must fail, thanks to the above item (2) and Theorem \ref{thm-green-sob}.
    \item[(4)] Finally, the on-diagonal upper bound of heat kernel in Corollary \ref{cor} (a) is also optimal in the sense that for any (small) $\epsilon>0$, the corresponding estimate with exponent of $t$ replaced by $\frac{nq}{q-n}-\epsilon$ is no longer true. To see this, let's recall a classical result (see e.g. \cite[Theorem 4.2.1]{Z}) that for a fixed number $s>1$, an on-diagonal upper bound of heat kernel $V_\omega H_\omega(x,x,t)\le \frac{C}{t^s}$ for any $t>0$ implies a Sobolev-type inequality with exponent $\frac{2s}{s-1}$. Actually, if one has $V_\omega H_\omega(x,x,t)\le \frac{C}{t^s}$, then $V_\omega H_\omega(x,y,t)\le \frac{C}{t^s}$ holds for any $x,y\in X$ and $t>0$, and hence we can apply \cite[Theorem 4.2.1, (III)$\Rightarrow$(I) ]{Z} to conclude the Sobolev-type inequality with exponent $\frac{2s}{s-1}$. Elementarily note that for $s<\frac{nq}{q-n}$, we have $\frac{2s}{s-1}>\frac{2nq}{nq-q+n}$.
    \end{itemize}
    \end{rem}

     \section*{Acknowledgments}
We are grateful to Zhenlei Zhang and Bin Zhou for helpful discussions. This work is partially supported by National Natural Science Foundation of China (No. 12371057), Natural Science Foundation of Hunan Province (No. 2024JJ2006) and The Science and Technology Innovation Program of Hunan Province (No. 2023RC3096).


\begin{thebibliography}{99}

\bibitem{B}R.~J. Berman, From Monge-Amp\`ere equations to envelopes and geodesic rays in the zero temperature limit, Math. Z. {\bf 291} (2019), no. 1-2, 365-394

\bibitem{C} G. Carron, In\'egalit\'es isop\'erim\'etriques de Faber-Krahn et cons\'equences, Publications de l'lnstitut Fourier, 220, 1992

\bibitem{Ch} J. Cheeger, Degeneration of Riemannian Metrics under Ricci Curvature Bounds, Lezioni Fermiane (Fermi Lectures), Scuola Normale Superiore, Pisa, 2001

\bibitem{CL} S. Y. Cheng and P. Li, Heat kernel estimates and lower bound of eigenvalues, Comment. Math. Helv. 56 (1981), no. 3, 327-338

\bibitem{DNV} T.D. Do, D.-B. Nguyen and D.-V. Vu, Non-collapsing volume estimate for local K\"{a}hler metrics in big cohomology classes, Math. Z. 312 (2026), no. 4, Paper No. 113, 29 pp

\bibitem{Gra} L. Grafakos, Classical Fourier Analysis, Third Edition, GTM 249, Springer, New York, 2014

\bibitem{GGZ} V. Guedj, H. Guenancia and A. Zeriahi, Diameter of K\"{a}hler current, J. Reine Angew. Math. (Crelle’s Journal) 820 (2025), 115-152

\bibitem{GL} V. Guedj and C. H. Lu, Quasi-plurisubharmonic envelopes 1: uniform estimates on K\"{a}hler manifolds, J. Eur. Math. Soc. 27, 1185-1208 (2025)

\bibitem{GT} V. Guedj and T. D. T\^o, K\"{a}hler families of Green's functions, J. \'Ec. polytech. Math. 12 (2025), 319-339

\bibitem{GPS} B. Guo, D.H. Phong and J. Sturm, Green's functions and complex Monge-Amp\`ere equations, J. Differential Geom. {\bf 127} (2024), no. 3, 1083-1119

\bibitem{GPSS1} B. Guo, D.H. Phong, J. Song and J. Sturm. Diameter estimates in K\"{a}hler geometry, Comm. Pure Appl. Math. {\bf 77} (2024), no. 8, 3520-3556

\bibitem{GPSS2} B. Guo, D.H. Phong, J. Song and J. Sturm. Diameter estimates in K\"{a}hler geometry II: removing the small degeneracy assumption, Math. Z. 308, 43 (2024)

\bibitem{GPSS3} B. Guo, D.H. Phong, J. Song and J. Sturm. Sobolev inequalities on K\"{a}hler spaces, arXiv:2311.00221

\bibitem{GPT} B. Guo, D.H. Phong and F. Tong, On $L^\infty$ estimates for complex Monge-Amp\`ere equations, Ann. Math. (2) 198 (2023), no. 1, 393-418

\bibitem{GPTW} B. Guo, D.H. Phong, F. Tong and C. Wang, On $L^\infty$ estimates for Monge-Amp\`ere and Hessian equations on nef classes, Anal. PDE 17 (2024), no. 2, 749-756

\bibitem{GS} B. Guo and J. Song, Local noncollapsing for complex Monge-Amp\`ere equations, J. Reine Angew. Math. {\bf 793} (2022), 225-238

\bibitem{H} E. Hebey, Nonlinear analysis on manifolds: Sobolev spaces and inequalities, Courant Lecture Notes in Mathematics, 5, New York Univ., Courant Inst. Math. Sci., New York, 1999 Amer. Math. Soc., Providence, RI, 1999

\bibitem{Le} N.Q. Le, Analysis of Monge-Amp\`ere equations, Grad. Stud. Math. 240, American Mathematical Society, Providence, RI, 2024

\bibitem{Li} P. Li, Geometric Analysis, Cambridge Stud. Adv. Math., 134, Cambridge University Press, Cambridge, 2012

\bibitem{Liuj} J. Liu, On relative $L^\infty$ estimate for complex Monge-Amp\`ere equation, arXiv:2410.04393 (2024)

\bibitem{NV} D.-B. Nguyen and D.-T. Vu, Uniform diameter estimates for K\"ahler metrics in big cohomology classes, arXiv:2410.18532

\bibitem{T} G. Tian, On K\"ahler-Einstein metrics on certain K\"ahler manifolds with $C_1(M)>0$, Invent. Math. {\bf 89} (1987), no. 2, 225-246

\bibitem{TW} G.J. Tian and X.J. Wang, A class of Sobolev type inequalities, Methods Appl. Anal. {\bf 15} (2008), no. 2, 263-276

\bibitem{V} D.-V. Vu, Uniform diameter and non-collapsing estimates for K\"{a}hler metrics, J. Geom. Anal. 36 (2026), no. 2, Paper No. 75, 34 pp

\bibitem{WZ} J. Wang and B. Zhou, Sobolev inequalities and regularity of the linearized complex Monge-Amp\`ere and Hessian equations, Trans. Amer. Math. Soc. {\bf 378} (2025), no. 1, 447-475

\bibitem{Y}S.-T. Yau, On the Ricci curvature of a compact K\"{a}hler manifold and the complex Monge-Amp\`{e}re equation, I, Comm. Pure Appl. Math. 31 (1978), 339-411

\bibitem{ZlZzl} L. Zhang and Z.L. Zhang, Complex Monge-Amp\`ere equation in Orlicz space and diameter bound, arXiv:2601.09893

\bibitem{Z}Q.S. Zhang, Sobolev inequalities, heat kernels under Ricci flow and the Poincar\'e conjecture, CRC-Press, Boca Raton, FL, 2011

\bibitem{ZwZy} W. Zhang and Y.S. Zhang, Green's function estimates for compact K\"{a}hler manifolds and applications, arXiv:2508.13646v2
\end{thebibliography}
\end{document}